\documentclass[tikz,border=10pt]{article}
\UseRawInputEncoding
\usepackage[english]{babel}
\usepackage[utf8x]{inputenc}
\usepackage[T1]{fontenc}

\usepackage{tikz}
\usepackage{pgflibraryarrows}
\usepackage{pgflibrarysnakes}
\usetikzlibrary{decorations.markings}

\usepackage[a4paper,top=3cm,bottom=2cm,left=3cm,right=3cm,marginparwidth=1.75cm]{geometry}

\usepackage{amsmath}
\usepackage{graphicx}
\usepackage[colorinlistoftodos]{todonotes}
\usepackage[colorlinks=true, allcolors=blue]{hyperref} \usepackage{mathrsfs}
\usepackage{amssymb}
\usepackage{lineno}
\usepackage{bbding}
\usepackage{fancyhdr}\usepackage{graphics}
\usepackage{titlesec}\usepackage{titletoc}
\usepackage{indentfirst}
\usepackage{amsmath}\usepackage{amsfonts}\usepackage{latexsym}
\usepackage{eucal}\usepackage{epsfig}\usepackage{mathrsfs,amsmath,amsthm}
\usepackage{geometry}

\newtheorem{theorem}{Theorem}[section]
 \newtheorem{lemma}[theorem]{Lemma}
 \newtheorem{proposition}[theorem]{Proposition}

 \newcount\refno
\def\CC{{\mathbb C}}

 \def\RR{{\mathbb R}}
 \def\NN{{\mathbb N}}
 \def\SS{{\mathbb S}}
 
 \def\ZZ{{\mathbb Z}}
\def\SD{{\mathscr D}}
 \def\SF{{\mathscr F}}
 
 \title{\bf Ces\`{a}ro operator on  Hardy spaces associated with the Dunkl setting\,($\frac{2\lambda}{2\lambda+1}<p<\infty$)
 \thanks{1} \footnote{E-mail:
huzhuoran010@163.com[ZhuoRan Hu].}}

\author{{ZhuoRan Hu}\\
{\small Department of Mathematics, Capital Normal University}\\
{\small Beijing 100048, China}}

\begin{document}

\maketitle \setcounter{page}{1} \pagestyle{myheadings}
 \markboth{Hu}{Ces\`{a}ro operator on  Hardy spaces associated with the Dunkl setting\,($\frac{2\lambda}{2\lambda+1}<p<\infty$)}

\begin{abstract}
For  $p>\frac{2\lambda}{2\lambda+1}$ with $\lambda>0$, the Hardy spaces $H_{\lambda}^{p}(\RR^{2}_+)$  associated with the Dunkl transform $\SF_{\lambda}$ and the Dunkl operator $D_x$ on the line, where  $D_xf(x)=f'(x)+\frac{\lambda}{x}[f(x)-f(-x)]$, is the set of  function $F=u+iv$ on the upper half plane $\RR_+^2=\big\{(x, y): y>0\big\}$, satisfying  the $\lambda$-Cauchy-Riemann equations: $D_xu-\partial_y v=0, \partial_y u +D_xv=0$, and $\sup_{y>0}\int_{\RR}|F(x, y)||x|^{2\lambda}dx<0$.  In this paper, we will study  the boundedness of Ces\`{a}ro operator on $H_{\lambda}^{p}(\RR^{2}_+)$. We will prove the following inequality
$$ \|C_{\alpha}f\|_{H_{\lambda}^p(\RR_+^2)}\leq C\|f\|_{H_{\lambda}^p(\RR_+^2)},$$
for $\frac{2\lambda}{2\lambda+1}< p<\infty$, where C is dependent on $\alpha$, $p$, $\lambda$, and the average function for the Ces\`{a}ro operator $C_{\alpha}$ is   $\phi_{\alpha}(t)=\alpha(1-t)^{\alpha-1}$ with $\alpha>0$.

\vskip .2in
 \noindent
 {\bf 2000 MS Classification:}
 \end{abstract}

\setcounter{page}{1}

\section{Introduction}
The Hausdorff operators on the real Hardy spaces $H^p(\RR)$ was  initially studied by  Kanjin in \cite{Ka}.   Miyachi studied the  Boundedness of the Ces\`{a}ro Operator on the Real Hardy Spaces in \cite{Mi1},  and   the Ces\`{a}ro  operator on the   $H^p$ on the Unit Disk was studied in \cite{Xi}. A brief history of the study of the Ces\`{a}ro operator and  the Hausdorff operators can be found in \cite{Mi1}, \cite{Ka},\cite{Li},\cite{Li2} and  \cite{Xi}. The Ces\`{a}ro operator  in this paper is  defined as in \cite{Mi1}.

The purpose of this paper is to consider the boundedness of Ces\`{a}ro operator on the  Hardy spaces associated with the Dunkl setting on the upper half plane $H_{\lambda}^{p}(\RR^{2}_+)$,\,($\frac{2\lambda}{2\lambda+1}<p\leq\infty$). More details associated with the Dunkl setting and $H_{\lambda}^{p}(\RR^{2}_+)$ will be introduced in Section\,\ref{ce Hp2}. One difference about $H_{\lambda}^{p}(\RR^{2}_+)$ and $H^{p}(\RR^{2}_+)$ is that, $H^{p}(\RR)$ related to $H^{p}(\RR^{2}_+)$ is a Homogeneous Hardy spaces, that we have an atomic decomposition  for any function in  $H^{p}(\RR)$, but similar properties is not known on the  $H_{\lambda}^{p}(\RR^{2}_+)$ including the related real $H_\lambda^{p}(\RR)$. Thus we could not use the same way in \cite{Mi1} to prove the boundedness of Ces\`{a}ro operator on $H_{\lambda}^{p}(\RR^{2}_+)$.

Our  \textbf{main  result} is Theorem\,\ref{cesa0}. we will prove the following inequality
$$ \|C_{\alpha}f\|_{H_{\lambda}^p(\RR_+^2)}\leq C\|f\|_{H_{\lambda}^p(\RR_+^2)},$$
for $\frac{2\lambda}{2\lambda+1}< p<\infty$, where C is a constant dependent on $\alpha$, $p$ and $\lambda$.

Similar result like Theorem\,\ref{cesa0}  could not be extended to the case $\RR^d$ where $d\geq3$, because  many properties of $H_{\lambda}^p(\RR_+^d)$ in \cite{J.P. Anker} and \cite{J.P. Anker2} remain unknown.

 \textbf{Note}: We use $A\lesssim B$ to denote the estimate $|A|\leq CB$ for some  absolute universal constant $C>0$, which may vary from line to line.
$A\gtrsim B$ to denote the estimate $|A|\geq CB$ for some absolute universal constant $C>0$.
 $A\approx B$ or $A\sim B$ to denote the estimate $|A|\leq C_1B$, $|A|\geq C_2B$ for some absolute universal constant $C_1, C_2$.

\section{Ces\`{a}ro operator on the spaces $L^p(\RR^{N+1}_+, d\mu(\mathbf{x}))$ }\label{ce Lp}
In this Section, we will discuss the Ces\`{a}ro operator on $L^p(\RR^{N+1}_+, d\mu(\mathbf{x}))$, which is the set of vector function $\vec{f}(\vec{x})=(u_0(\vec{x}), u_1(\vec{x}), u_2(\vec{x})\cdots u_N(\vec{x}))$ defined on $\RR_+^{N+1}$ with $\vec{x}=(\mathbf{x}, y)=(x_1,\cdots,x_N, y)\in\RR_+^{N+1}$, where $y>0$ and $\mathbf{x}\in\RR^N$, $\mathbf{x}=(x_1,\cdots,x_N)\in\RR^{N}$. Let the measure $d\mu(\mathbf{x})$ to denote as
$d\mu(\mathbf{x})=|x_1|^{2\lambda_1} |x_2|^{2\lambda_2} \cdots |x_N|^{2\lambda_N} dx_1 dx_2 \cdots dx_N,\ \ \hbox{where}\,0<\lambda_i<\infty\ \hbox{for}\, 1\leq i\leq N.$

Then  $L^p(\RR^{N+1}_+, d\mu(\mathbf{x}))$ can be defined as the set of vector function $\vec{f}(\vec{x})$ satisfying
 $$\|\vec{f}\|_{L^p(\RR^{N+1}_+, d\mu(\mathbf{x}))}=\sup\limits_{y>0}\Big(\int_{\RR^N}|\vec{f}(\vec{x})|^pd\mu(\mathbf{x})\Big)^{1/p}<\infty,\ \hbox{where},\ \ 0<p<\infty.$$
 $L_{\mu}^p(\RR^N)$ denote as the set of measurable function $f(\mathbf{x})$ on $\RR^N$ satisfying $ \|f\|_{L_{\mu}^p(\RR^N)}=\Big(c_{\lambda}\int_{\RR^N}|f(\mathbf{x})|^pd\mu(\mathbf{x})\Big)^{1/p}$ $<\infty$,  with
$c_{\lambda}^{-1}=2^{\lambda+1/2}\Gamma(\lambda+1/2)$. To avoid confusion, we will use the symbol $\vec{f}$ to denote as a vector function, and $f$ to
denote as the ordinary measurable function.

Let $\alpha>0$, we write $\phi_{\alpha}(t)=\alpha(1-t)^{\alpha-1}$ for $0<t<1$.    For $\overrightarrow{f}(\mathbf{x},y)=(u_0(\mathbf{x}, y),u_1(\mathbf{x}, y),\cdots,u_N(\mathbf{x}, y)) \in L^p(\RR^{N+1}_+, d\mu(\mathbf{x}))$,  we  define the Ces\`{a}ro operator $C_{\alpha}$ by the form  as
$$(C_{\alpha}u_j)(\mathbf{x}, y)=\int_0^1 t^{-1}u_j\left(t^{-1}\mathbf{x}, t^{-1}y\right)\phi_{\alpha}(t) dt,\ \ \hbox{for}, \ \,0\leq j\leq N$$
where $ \mathbf{x}\in\RR^{N+1}_+, t^{-1}\mathbf{x}=(t^{-1}x_1, t^{-1}x_2,\cdots, t^{-1}x_N)$.
Then $C_{\alpha}\overrightarrow{f}$ and $\left|C_{\alpha}\overrightarrow{f}\right|$ can be given by
$$C_{\alpha}\overrightarrow{f}=(C_{\alpha}u_0, C_{\alpha}u_1,\cdots, C_{\alpha}u_N),$$
$$\left|C_{\alpha}\overrightarrow{f}\right|=\bigg((C_{\alpha}u_0)^2+ (C_{\alpha}u_1)^2+\cdots+ (C_{\alpha}u_N)^2\bigg)^{1/2}.$$
It is easy to verify the following Proposition\,\ref{ko} and Proposition\,\ref{Minkow}:
\begin{proposition}\label{ko}
For $\overrightarrow{f}(\mathbf{x},y)=(u_0(\mathbf{x}, y),u_1(\mathbf{x}, y),\cdots,u_N(\mathbf{x}, y)) \in L^p(\RR^{N+1}_+, d\mu(\mathbf{x}))$, where $0<p<\infty,$
$$\left|C_{\alpha}\overrightarrow{f}\right|\leq C_{\alpha}\left|\overrightarrow{f}\right|.$$

\end{proposition}

\begin{proposition}\label{Minkow}[Minkowski's inequality] For $1\leq p<\infty$, $\ \mathbf{x}=(x_1,\cdots,x_N)$,
if $$\int_0^1\left[\int_{\RR^N}|f(\mathbf{x}, t)|^pd\mu(\mathbf{x})\right]^{1/p}dt=M<\infty$$
then
$$\left[\int_{\RR^N}\left|\int_0^1 f(\mathbf{x}, t)dt\right|^pd\mu(\mathbf{x})\right]^{1/p}\leq\int_0^1\left[\int_{\RR^N}|f(\mathbf{x}, t)|^pd\mu(\mathbf{x})\right]^{1/p}dt=M<\infty.$$
\end{proposition}
By Proposition\,\ref{Minkow}, we could prove the following Proposition\,\ref{cesar1}
\begin{proposition}\label{cesar1}
Ces\`{a}ro operator $C_{\alpha}$ is a bounded linear operator on $L^p(\RR^{N+1}_+, d\mu(\mathbf{x}))$ for $1\leq p<\infty$:
$$ \|C_{\alpha}\overrightarrow{f}\|_{L^p(\RR^{N+1}_+, d\mu(\mathbf{x}))}\leq C\|\overrightarrow{f}\|_{L^p(\RR^{N+1}_+, d\mu(\mathbf{x}))}.$$
C is dependent on $\alpha$, $p$ and $d\mu$.

\end{proposition}
\begin{proof}
From Proposition\,\ref{ko} and the Minkowski's inequality in Proposition\,\ref{Minkow}, we could conclude:
\begin{eqnarray*}
\|C_{\alpha}\overrightarrow{f}\|_{L^p(\RR^{N+1}_+, d\mu(\mathbf{x}))}
&=&\sup_{y>0}\left[\int_{\RR^n}\left|C_{\alpha}\overrightarrow{f}(\mathbf{x}, y)\right|^pd\mu(\mathbf{x})\right]^{1/p}
\\&=&\sup_{y>0}\left[\int_{\RR^n}\left|\int_0^1 t^{-1}\overrightarrow{f}\left(t^{-1}\mathbf{x}, t^{-1}y\right)\phi_{\alpha}(t) dt\right|^pd\mu(\mathbf{x})\right]^{1/p}
\\&\leq&\sup_{y>0}\left[\int_{\RR^n}\left|\int_0^1 \left|\overrightarrow{f}\left(t^{-1}\mathbf{x}, t^{-1}y\right)\right|t^{-1}\phi_{\alpha}(t) dt\right|^pd\mu(\mathbf{x})\right]^{1/p}
\\&\leq&\left[\int_0^1\sup_{y>0}\left|\int_{\RR^n} |\overrightarrow{f}\left(t^{-1}\mathbf{x}, t^{-1}y\right)|^pd\mu(\mathbf{x})\right|^{1/p} t^{-1}\phi_{\alpha}(t) dt\right]
\\&\leq& \left[\int_0^1\|\overrightarrow{f}\|_{L^p(\RR^{N+1}_+, d\mu(\mathbf{x}))} t^{-1}t^{\beta/p}\phi_{\alpha}(t) dt\right]
\\&\leq& C\|\overrightarrow{f}\|_{L^p(\RR^{N+1}_+, d\mu(\mathbf{x}))},\ \ \ \ \hbox{where\, $\beta=N+\sum_{i=1}^{N}2\lambda_i$.}
\end{eqnarray*}

\end{proof}

\begin{lemma}\label{cesa2}
For $0<p\leq1$, $i\in\NN$, if $\sum_i|a_i|^p<\infty$, then
$$\left(\sum_i|a_i|\right)^p\leq \sum_i|a_i|^p .$$
\end{lemma}

\begin{proof}
Without loss of generality, we may assume that
$$\sum_i|a_i|^p=1.$$
In this case, we have $|a_i|\leq1$ for any $i\in\NN$, so
$$\left(\sum_i|a_i|\right)\leq \sum_i|a_i|^p|a_i|^{1-p} \leq \sum_i|a_i|^p=1 .$$
Then we could obtain:
$$\left(\sum_i|a_i|\right)^p\leq \sum_i|a_i|^p .$$
This proves the Lemma.
\end{proof}

\begin{lemma}\label{cesa3}
For $0<p\leq1$, $k\in\ZZ$, $k\leq-1$, $\mathbf{x}\in\RR^N$ there exists $\xi_k\in [2^{k-1}, 2^k]$ and $\xi'_k\in [1-2^k, 1-2^{k-1}]$ such that the following holds:
$$\left|\int_0^1 t^{-1}\overrightarrow{f}\left(t^{-1}\mathbf{x}, t^{-1}y\right)\phi_{\alpha}(t) dt\right|^p\leq C_\alpha\sum_{k=-\infty}^{-1}|\overrightarrow{f}\left(\xi_k^{-1}\mathbf{x}, \xi_k^{-1}y\right)|^p +C_\alpha\sum_{k=-\infty}^{-1}2^{ kp \alpha}|\overrightarrow{f}\left((\xi'_k)^{-1}\mathbf{x}, (\xi'_k)^{-1}y\right)|^p ,$$
where $C_\alpha$ is a constant independent on $\mathbf{x}$, $y$,  $k$ and $\overrightarrow{f}$.
\end{lemma}

\begin{proof}By Proposition\,\ref{ko} and Lemma\,\ref{cesa2},
\begin{eqnarray}\label{f1}
\left|\int_0^1 t^{-1}\overrightarrow{f}\left(t^{-1}\mathbf{x}, t^{-1}y\right)\phi_{\alpha}(t) dt\right|^p\leq\left|\int_0^{1/2} t^{-1}|\overrightarrow{f}\left(t^{-1}\mathbf{x}, t^{-1}y\right)|\phi_{\alpha}(t) dt\right|^p+\left|\int_{1/2}^1 t^{-1}|\overrightarrow{f}\left(t^{-1}\mathbf{x}, t^{-1}y\right)|\phi_{\alpha}(t) dt\right|^p.
\end{eqnarray}
By Lemma\,\ref{cesa2}, we could conclude:
\begin{eqnarray*}
\left|\int_0^{1/2} t^{-1}|\overrightarrow{f}\left(t^{-1}\mathbf{x}, t^{-1}y\right)|\phi_{\alpha}(t) dt\right|^p&=&\left|\sum_{k=-\infty}^{-1}\int_{2^{k-1}}^{2^k} t^{-1}|\overrightarrow{f}\left(t^{-1}\mathbf{x}, t^{-1}y\right)|\phi_{\alpha}(t) dt\right|^p
\\&\leq&\sum_{k=-\infty}^{-1}\left|\int_{2^{k-1}}^{2^k} t^{-1}|\overrightarrow{f}\left(t^{-1}\mathbf{x}, t^{-1}y\right)|\phi_{\alpha}(t) dt\right|^p.
\end{eqnarray*}
We could deduce that there exists $\xi_k\in [2^{k-1}, 2^k]$ and a constant $C$ ($C=200$ for example) independent on $k$, $\overrightarrow{f}$,  $\mathbf{x}$, and $y$, such that the
following holds:
\begin{eqnarray*}
\int_{2^{k-1}}^{2^k} t^{-1}|\overrightarrow{f}\left(t^{-1}\mathbf{x}, t^{-1}y\right)|\phi_{\alpha}(t) dt&\leq& C \xi_k^{-1}|\overrightarrow{f}\left(\xi_k^{-1}\mathbf{x}, \xi_k^{-1}y\right)|\int_{2^{k-1}}^{2^k} \phi_{\alpha}(t) dt
\\ \nonumber&\leq& (1/2)^{\alpha}C |\overrightarrow{f}\left(\xi_k^{-1}\mathbf{x}, \xi_k^{-1}y\right)|\\ \nonumber&\lesssim_\alpha& |\overrightarrow{f}\left(\xi_k^{-1}\mathbf{x}, \xi_k^{-1}y\right)|.
\end{eqnarray*}
Then
\begin{eqnarray}\label{f2}
\left|\int_0^{1/2} t^{-1}|\overrightarrow{f}\left(t^{-1}\mathbf{x}, t^{-1}y\right)|\phi_{\alpha}(t) dt\right|^p \nonumber&\leq&\sum_{k=-\infty}^{-1}\left|\int_{2^{k-1}}^{2^k} t^{-1}|\overrightarrow{f}\left(t^{-1}\mathbf{x}, t^{-1}y\right)|\phi_{\alpha}(t) dt\right|^p \\&\leq&C\sum_{k=-\infty}^{-1}|\overrightarrow{f}\left(\xi_k^{-1}\mathbf{x}, \xi_k^{-1}y\right)|^p.
\end{eqnarray}
In the same way we could conclude that:
\begin{eqnarray*}
\left|\int_{1/2}^1 t^{-1}|\overrightarrow{f}\left(t^{-1}\mathbf{x}, t^{-1}y\right)|\phi_{\alpha}(t) dt\right|^p&=&\left|\sum_{k=-\infty}^{-1}\int_{1-2^{k}}^{1-2^{k-1}} t^{-1}|\overrightarrow{f}\left(t^{-1}\mathbf{x}, t^{-1}y\right)|\phi_{\alpha}(t) dt\right|^p
\\&\leq&\sum_{k=-\infty}^{-1}\left|\int_{1-2^{k}}^{1-2^{k-1}} t^{-1}|\overrightarrow{f}\left(t^{-1}\mathbf{x}, t^{-1}y\right)|\phi_{\alpha}(t) dt\right|^p.
\end{eqnarray*}
We could deduce that there exists $\xi'_k\in [1-2^k, 1-2^{k-1}]$, such that the
following holds:
\begin{eqnarray*}
\int_{1-2^k}^{1-2^{k-1}} t^{-1}|\overrightarrow{f}\left(t^{-1}\mathbf{x}, t^{-1}y\right)|\phi_{\alpha}(t) dt&\leq& C |\overrightarrow{f}\left((\xi'_k)^{-1}\mathbf{x}, (\xi'_k)^{-1}y\right)|\int_{1-2^k}^{1-2^{k-1}} t^{-1}\phi_{\alpha}(t) dt
\\ \nonumber&\leq&C2^{k\alpha} |\overrightarrow{f}\left((\xi'_k)^{-1}\mathbf{x}, (\xi'_k)^{-1}y\right)|.
\end{eqnarray*}
Thus
\begin{eqnarray}\label{f3}
\left|\int_{1/2}^1 t^{-1}|\overrightarrow{f}\left(t^{-1}\mathbf{x}, t^{-1}y\right)|\phi_{\alpha}(t) dt\right|^p&=&\left|\sum_{k=-\infty}^{-1}\int_{1-2^{k}}^{1-2^{k-1}} t^{-1}|\overrightarrow{f}\left(t^{-1}\mathbf{x}, t^{-1}y\right)|\phi_{\alpha}(t) dt\right|^p
\\ \nonumber  &\leq&\sum_{k=-\infty}^{-1}\left|\int_{1-2^{k}}^{1-2^{k-1}} t^{-1}|\overrightarrow{f}\left(t^{-1}\mathbf{x}, t^{-1}y\right)|\phi_{\alpha}(t) dt\right|^p
\\ \nonumber &\leq& C\sum_{k=-\infty}^{-1}2^{kp \alpha}|\overrightarrow{f}\left((\xi'_k)^{-1}\mathbf{x}, (\xi'_k)^{-1}y\right)|^p .
\end{eqnarray}
Then we could obtain the Lemma by Formulas\,(\ref{f1},\,\ref{f2},\,\ref{f3}).
\end{proof}

\begin{proposition}\label{cesa5}
For $0<p\leq1$, the Ces\`{a}ro operator $C_{\alpha}$ is bounded on  $L^p(\RR^{N+1}_+, d\mu(\mathbf{x}))$ spaces:
$$ \|C_{\alpha}\overrightarrow{f}\|_{L^p(\RR^{N+1}_+, d\mu(\mathbf{x}))}\leq C\|\overrightarrow{f}\|_{L^p(\RR^{N+1}_+, d\mu(\mathbf{x}))}.$$
C is a constant dependent on $\alpha$, $p$ and $d\mu$.
\end{proposition}

\begin{proof}
By Lemma\,\ref{cesa3}, we have that:
\begin{eqnarray*}
& &\|C_{\alpha}\overrightarrow{f}\|_{L^p(\RR^{N+1}_+, d\mu(\mathbf{x}))}^p
\\&\leq& C\sup_{y>0}\int_{\RR^N}\sum_{k=-\infty}^{-1}|\overrightarrow{f}\left(\xi_k^{-1}\mathbf{x}, \xi_k^{-1}y\right)|^pd\mu(\mathbf{x}) +C\sup_{y>0}\int_{\RR^N}\sum_{k=-\infty}^{-1}2^{ kp \alpha}|\overrightarrow{f}\left((\xi'_k)^{-1}\mathbf{x}, (\xi'_k)^{-1}y\right)|^pd\mu(\mathbf{x})
\\&\leq&C\sum_{k=-\infty}^{-1}\left(2^{\beta k}+2^{kp\alpha}\right)\|\overrightarrow{f}\|_{L^p(\RR^{N+1}_+, d\mu(\mathbf{x}))}^p,\ \ \ \ \hbox{where\, $\beta=N+\sum_{i=1}^{N}2\lambda_i$.}
\end{eqnarray*}

This proves the proposition.
\end{proof}

From Proposition\,\ref{cesar1} and Proposition\,\ref{cesa5}, we could obtain the following theorem:
\begin{theorem}\label{cesar2}
For $0<p<\infty$, the Ces\`{a}ro operator $C_{\alpha}$ is bounded on  $L^p(\RR^{N+1}_+, d\mu(\mathbf{x}))$ spaces:
$$ \|C_{\alpha}\overrightarrow{f}\|_{L^p(\RR^{N+1}_+, d\mu(\mathbf{x}))}\leq C\|\overrightarrow{f}\|_{L^p(\RR^{N+1}_+, d\mu(\mathbf{x}))}.$$
C is a constant dependent on $\alpha$, $p$ and $d\mu$. Thus the Ces\`{a}ro operator $C_{\alpha}$ is well defined on   $L^p(\RR^{N+1}_+, d\mu(\mathbf{x}))$ spaces.
\end{theorem}

\section{Ces\`{a}ro operator on the spaces $H_{\lambda}^{p}(\RR^{2}_+)$ }\label{ce Hp2}
In this section, we will discuss the {Ces\`{a}ro operator on the spaces $H_{\lambda}^{p}(\RR^{2}_+)$. First, we will introduce some basic concept associated with the Dunkl setting and $H_{\lambda}^{p}(\RR^{2}_+)$, then, with the Theorem\,\ref{cesar2} we obtained in previous section, we will prove the boundedness of  {Ces\`{a}ro operator on the spaces $H_{\lambda}^{p}(\RR^{2}_+)$.

\subsection{The $\lambda$-translation and the $\lambda$-convolution}
In\,\cite{ZhongKai Li 3}, a theory about the Hardy spaces on the half-plane $\RR^2_+=\left\{(x, y): x\in\RR, y>0\right\}$ associated with the Dunkl transform on the line $\RR$ is developed. For $0<p<\infty$, $L_{\lambda}^p(\RR)$ is the set of measurable functions satisfying
$ \|f\|_{L_{\lambda}^p}=\Big(c_{\lambda}\int_{\RR}|f(x)|^p|x|^{2\lambda}dx\Big)^{1/p}$ $<\infty$ with
$c_{\lambda}^{-1}=2^{\lambda+1/2}\Gamma(\lambda+1/2)$,
and $p=\infty$ is the usual $L^\infty(\RR)$ space.
For $\lambda\geq0$, The Dunkl operator on the line is\,(cf. \cite{Du3},\cite{Du6}):
$$D_xf(x)=f'(x)+\frac{\lambda}{x}[f(x)-f(-x)]$$
involving a reflection part.  We assume $\lambda>0$ in
what follows.
The Dunkl transfrom  for $f\in L_{\lambda}^1(\RR)$ is given by:
\begin{eqnarray}\label{fourier}
(\SF_{\lambda}f)(\xi)=c_{\lambda}\int_{\RR}f(x)E_\lambda(-ix\xi)|x|^{2\lambda}dx,\quad
\xi\in\RR ,
\end{eqnarray}
where $E_{\lambda}(-ix\xi)$ is the Dunkl kernel
$$E_{\lambda}(iz)=j_{\lambda-1/2}(z)+\frac{iz}{2\lambda+1}j_{\lambda+1/2}(z),\ \  z\in\CC$$
and $j_{\alpha}(z)$ is the normalized Bessel function
$$j_{\alpha}(z)=2^{\alpha}\Gamma(\alpha+1)\frac{J_{\alpha}(z)}{z^{\alpha}}=\Gamma(\alpha+1)\sum_{n=0}^{\infty}\frac{(-1)^n(z/2)^{2n}}{n!\Gamma(n+\alpha+1)} .$$
Since $j_{\lambda-1/2}(z)=\cos z$, $j_{\lambda+1/2}(z)=z^{-1}\sin z$. It follows that $E_0(iz)=e^{iz}$, and $\SF_{0}$ agrees with the usual Fourier transform. $E_\lambda(iz)$ can  also be represented by the following \cite{Ro1}
\begin{eqnarray}\label{D-kernel-3}
E_\lambda(iz)=c_{\lambda}'\int_{-1}^1e^{izt}(1+t)(1-t^2)^{\lambda-1}dt,
\ \ \ \ \ \
c'_{\lambda}=\frac{\Gamma(\lambda+1/2)}{\Gamma(\lambda)\Gamma(1/2)}.
\end{eqnarray}
where $E_{\lambda}(ix\xi)$ satisfies
\begin{eqnarray}\label{D-kernel-2}
D_x[E_{\lambda}(ix\xi)]=i\xi E_{\lambda}(ix\xi), \ \ \ \ \ \hbox{and} \
\ \ \ E_{\lambda}(ix\xi)|_{x=0}=1.
\end{eqnarray}
For $x,t,z\in\RR$, we set
$W_{\lambda}(x,t,z)=W_{\lambda}^0(x,t,z)(1-\sigma_{x,t,z}+\sigma_{z,x,t}+\sigma_{z,t,x})$,
where
$$
W_{\lambda}^0(x,t,z)=\frac{c''_{\lambda}|xtz|^{1-2\lambda}\chi_{(||x|-|t||,
|x|+|t|)}(|z|)} {[((|x|+|t|)^2-z^2)(z^2-(|x|-|t|)^2)]^{1-\lambda}},
$$
$c''_{\lambda}=2^{3/2-\lambda}\big(\Gamma(\lambda+1/2)\big)^2/[\sqrt{\pi}\,\Gamma(\lambda)]$,
and  $\sigma_{x,t,z}=\frac{x^2+t^2-z^2}{2xt}$ for $x,t\in
\RR\setminus\{0\}$,
and 0 otherwise. From\,\cite{Ro1}, we have

\begin{proposition}\label{Dunkl-kernel-P}
The Dunkl kernel $E_\lambda$ satisfies the following product formula:
\begin{eqnarray}\label{product-D-1}
E_\lambda(x\xi)E_\lambda(t\xi)=\int_{\RR}E_\lambda(z\xi)d\nu_{x,t}(z),\quad
x,t\in\RR \ ~\hbox{and} ~ \ \xi\in\CC,
\end{eqnarray}
where $\nu_{x,t}$ is a signed measure given by
$d\nu_{x,t}(z)=c_{\lambda}W_\lambda(x,t,z)|z|^{2\lambda}dz$
for $x,t\in \RR\setminus\{0\}$, $d\nu_{x,t}(z)=d\delta_x(z)$ for $t=0$,
 and $d\nu_{x,t}(z)=d\delta_t(z)$ for $x=0$.
\end{proposition}
If $t\neq0$, for an appropriate function $f$ on $\RR$, the  $\lambda$-translation is given by
\begin{eqnarray}\label{tau-1}
(\tau_tf)(x)=c_{\lambda}\int_{\RR}f(z)W_\lambda(x,t,z)|z|^{2\lambda}dz;
\end{eqnarray}
and if $t=0$, $(\tau_0f)(x)=f(x)$.
If $(\tau_tf)(x)$ is taken as a function of $t$ for a given $x$,
we may set $(\tau_tf)(0)=f(t)$ as  a complement.
An unusual fact is that
$\tau_t$ is not a positive operator in general \cite{Ro1}.
If $(x, t)\neq(0,0)$, an equivalent form of $(\tau_tf)(x)$ is given by\,\cite{Ro1},
\begin{eqnarray}\label{tau-2}
(\tau_tf)(x)=c'_{\lambda}\int_0^\pi \bigg(f_e(\langle
x,t\rangle_\theta)+ f_o(\langle
x,t\rangle_\theta)\frac{x+t}{\langle
x,t\rangle_\theta}\bigg)(1+\cos\theta)\sin^{2\lambda-1}\theta
d\theta
\end{eqnarray}
where $x,t\in\RR$, $f_e(x)=(f(x)+f(-x))/2$,
$f_o(x)=(f(x)-f(-x))/2$, $\langle
x,t\rangle_\theta=\sqrt{x^2+t^2+2xt\cos\theta}$.
For two appropriated function $f$ and $g$, their $\lambda$-convolution
 $f\ast_{\lambda}g$ is defined by
\begin{eqnarray}\label{convolution-10}
(f\ast_{\lambda}
g)(x)=c_{\lambda}\int_{\RR}(\tau_xf)(-t)g(t)|t|^{2\lambda}dt.
\end{eqnarray}
The properties of $\tau$ and $\ast_{\lambda}$ are listed as follows\,\cite{ZhongKai Li 3}:

\begin{proposition}\label{tau-convolution-a} {\rm(i)} \ If $f\in L_{\lambda,{\rm loc}}(\RR)$,
then for all $x,t\in\RR$, $(\tau_tf)(x)=(\tau_xf)(t)$, and
$(\tau_t\tilde{f})(x)=(\widetilde{\tau_{-t}f})(x)$, where $\widetilde{f}(x)=f(-x)$.

{\rm(ii)}  \ For all $1\le p\le\infty$ and $f\in L_{\lambda}^p(\RR)$,
$\|\tau_tf\|_{L^p_{\lambda}}\le 4\|f\|_{L^p_{\lambda}}$ with $t\in\RR$, and for $1\leq p<\infty$, $\lim_{t\rightarrow0}\|\tau_tf-f\|_{L^p_{\lambda}}=0$.

{\rm(iii)} \ If $f\in L^p_\lambda(\RR)$, $1\leq p\leq2$ and $t\in\RR$, then
$[\SF_{\lambda}(\tau_t
f)](\xi)=E_{\lambda}(it\xi)(\SF_{\lambda}f)(\xi)$ for almost every $\xi\in\RR$.

{\rm(iv)} \ For measurable $f,g$ on $\RR$, if $\iint
|f(z)||g(x)||W_{\lambda}(x,t,z)||z|^{2\lambda}|x|^{2\lambda}dzdx$ is convergent, we have
  $\langle\tau_tf,g\rangle_{\lambda}=\langle
f,\tau_{-t}g\rangle_{\lambda}$. In particular,
$\ast_{\lambda}$ is commutative.

{\rm(v)} \ {\rm(Young inequality)} \ If $p,q,r\in[1,\infty]$ and
$1/p+1/q=1+1/r$, then for  $f\in L^p_{\lambda}(\RR)$, $g\in
L^q_{\lambda}(\RR)$, we have $\|f\ast_{\lambda}g\|_{L^r_{\lambda}}\leq
4\|f\|_{L^p_{\lambda}} \|g\|_{L^q_{\lambda}}$.

{\rm(vi)} \ Assume that $p,q,r\in[1,2]$, and $1/p+1/q=1+1/r$. Then for $f\in L^p_{\lambda}(\RR)$ and $g\in
L^q_{\lambda}(\RR)$,
$[\SF_{\lambda}(f\ast_{\lambda}g)](\xi)=(\SF_{\lambda}f)(\xi)(\SF_{\lambda}g)(\xi)$. In particular,
$\ast_{\lambda}$ is associative in $L^1_{\lambda}(\RR)$

{\rm(vii)} \ If $f\in L^1_{\lambda}(\RR)$ and $g\in\SS(\RR)$, then
$f\ast_{\lambda}g\in C^{\infty}(\RR)$.

{\rm(viii)} \ If $f,g\in L_{\lambda,{\rm loc}}(\RR)$, ${\rm
supp}\,f\subseteq\{x:\, r_1\le|x|\le r_2\}$ and ${\rm
supp}\,g\subseteq\{x:\,|x|\le r_3\}$,$r_2>r_1>0$, $r_3>0$, then
${\rm supp}\, (f\ast_{\lambda}g)\subseteq\{x:\,r_1-r_3\le|x|\le
r_2+r_3\}$.
\end{proposition}

\begin{proposition}(\hbox{\cite{ZhongKai Li 3}, Corollary\,2.5(i),\,Lemma\,2.7,\,Proposition\,2.8(ii)})\label{D-translation-a}

{\rm (i)} \ If $f\in\SS(\RR)$ or $\SD(\RR)$, then for fixed
$t$, the function
$x\mapsto(\tau_tf)(x)$ is also in $\SS(\RR)$ or $\SD(\RR)$, and
\begin{eqnarray}\label{D-translation-10}
 D_t(\tau_tf(x))=D_x(\tau_tf(x))=[\tau_t(Df)](x).
\end{eqnarray}
If $f\in\SS(\RR)$ and $m>0$, a pointwise estimate is given by
\begin{eqnarray}\label{D-translation-2}
\big|(\tau_t|f|)(x)\big|\le
\frac{c_m(1+x^2+t^2)^{-\lambda}}{(1+||x|-|t||^2)^m}.\end{eqnarray}

{\rm (ii)} \ If $\phi\in L_\lambda^1(\RR)$ satisfies $(\SF_{\lambda}\phi)(0)=1$ and $\phi_{\epsilon}(x)=\epsilon^{-2\lambda-1}\phi(\epsilon^{-1}x)$ for
$\epsilon>0$, then for all $f\in X= L_{\lambda}^p(\RR)$, $ 1\leq p<\infty$, or $C_0(\RR)$, $\lim_{\epsilon\rightarrow0+}\|f\ast_{\lambda}\phi_{\epsilon}-f\|_X=0$.

\end{proposition}

\subsection{Some facts about $H_{\lambda}^{p}(\RR^{2}_+)$ }\label{ce Hp3}

When u and v  satisfy the $\lambda$-Cauchy-Riemann equations:
\begin{eqnarray}\label{a c r0}
\left\{\begin{array}{ll}
                                    D_xu-\partial_y v=0,&  \\
                                    \partial_y u +D_xv=0&
                                 \end{array}\right.
\end{eqnarray}
the function F(z)=F(x,y)=u(x,y)+iv(x,y)\,(z=x+iy)\, is said to be a $\lambda$-analytic function on the upper half plane $\RR^2_+$. It is easy to see that
$F=u+iv$ is $\lambda$-analytic on $\RR^2_+$ if and only if
$$T_{\bar{z}}F\equiv0,\ \ \ \hbox{with}\,T_{\bar{z}}=\frac{1}{2}(D_x+i\partial_y).$$
If $u$ and $v$ are $C^2$ functions and satisfy\,(\ref{a c r0}), then
\begin{eqnarray}\label{1}
(\triangle_{\lambda}u)(x, y)=0,\ \ \hbox{with}\,\triangle_{\lambda}=D_x^2+ \partial_y^2.
\end{eqnarray}
A $C^2$ function $u(x, y)$ satisfying Formula\,(\ref{1}) is said to be $\lambda$-harmonic.\,In\,\cite{ZhongKai Li 3}, the
Hardy space $H^p_\lambda(\RR^2_+)$ for $p>0$ is defined to be the set of
$\lambda$-analytic functions on $\RR^2_+$ satisfying
$$\|F\|_{H^p_\lambda(\RR^2_+)}=\sup\limits_{y>0}\left\{c_{\lambda}\int_{\RR}|F(x+iy)|^p|x|^{2\lambda}dx \right\}^{1/p}<+\infty.$$
When $p>\frac{2\lambda}{2\lambda+1}$, some basic conclusions on $H_{\lambda}^p(\RR_+^2)$  are obtained in  \cite{ZhongKai Li 3}, together with
the associated real Hardy space $H_{\lambda}^p(\RR)$ on the line $\RR$, the collection  of the real parts of boundary functions   of  $F\in H_{\lambda}^p(\RR_+^2)$.

The following  results about  $\lambda$-Poisson integral and Conjugate $\lambda$-Poisson integral are obtained in \cite{ZhongKai Li 3}.
For $1\leq p\leq\infty$, the $\lambda$-Poisson integral of $f\in L_\lambda^p(\RR)$ in the Dunkl setting is given by:
$$
(Pf)(x,y)=c_{\lambda}\int_{\RR}f(t)(\tau_xP_y)(-t)|t|^{2\lambda}dt=c_{\lambda}(f\ast_{\lambda}
P_y)(x),
\ \ \ \ \ \ \hbox{for} \ x\in\RR, \ y\in(0,\infty),
$$
where $(\tau_xP_y)(-t)$ is the $\lambda$-Poisson kernel with $P_y(x)=m_{\lambda}y(y^2+x^2)^{-\lambda-1}$,
$m_\lambda=2^{\lambda+1/2}\Gamma(\lambda+1)/\sqrt{\pi}$. Similarly,  the $\lambda$-Poisson integral for  $d\mu\in {\frak
B}_{\lambda}(\RR)$ can be given by $
(P(d\mu))(x,y)=c_{\lambda}\int_{\RR}(\tau_xP_y)(-t)|t|^{2\lambda}d\mu(t). $

\begin{proposition}\cite{ZhongKai Li 3}\label{Poisson-a} {\rm(i)} \
$\lambda$-Poisson kernel $(\tau_xP_y)(-t)$ can be represented by
\begin{eqnarray}\label{D-Poisson-ker-11}
(\tau_xP_y)(-t)=
\frac{\lambda\Gamma(\lambda+1/2)}{2^{-\lambda-1/2}\pi}\int_0^\pi\frac{y(1+{\rm
sgn}(xt)\cos\theta)
}{\big(y^2+x^2+t^2-2|xt|\cos\theta\big)^{\lambda+1}}\sin^{2\lambda-1}\theta
d\theta.
\end{eqnarray}

{\rm(ii)} \
The Dunkl transform of the function $P_y(x)$ is $(\SF_{\lambda}P_y)(\xi)=e^{-y|\xi|}$, and
$
(\tau_xP_y)(-t)=c_{\lambda}\int_{\RR}e^{-y|\xi|}E_{\lambda}(ix\xi)E_{\lambda}(-it\xi)|\xi|^{2\lambda}d\xi.
$
\end{proposition}
Relating to $P_y(x)$, the conjugate $\lambda$-Poisson integral for $f\in L^p_{\lambda}(\RR)$ is given by:
$$
(Qf)(x,y)=c_{\lambda}\int_{\RR}f(t)(\tau_xQ_y)(-t)|t|^{2\lambda}dt=c_{\lambda}(f\ast_{\lambda}
Q_y)(x),
\ \ \ \ \ \ \hbox{for} \ x\in\RR, \ y\in(0,\infty),
$$
where $(\tau_xQ_y)(-t)$ is the conjugate $\lambda$-Poisson kernel with $Q_y(x)=m_{\lambda}x(y^2+x^2)^{-\lambda-1}$.

\begin{proposition}\cite{ZhongKai Li 3}\label{conjugate-Poisson-a} {\rm(i)} \ \
The conjugate  $\lambda$-Poisson kernel $(\tau_xQ_y)(-t)$ can be represented by
\begin{eqnarray*}\label{D-conjugate-Poisson-ker-1}
(\tau_xQ_y)(-t)=
\frac{\lambda\Gamma(\lambda+1/2)}{2^{-\lambda-1/2}\pi}\int_0^\pi\frac{(x-t)(1+{\rm
sgn}(xt)\cos\theta)
}{\big(y^2+x^2+t^2-2|xt|\cos\theta\big)^{\lambda+1}}\sin^{2\lambda-1}\theta
d\theta.
\end{eqnarray*}

{\rm(ii)} \  The Dunkl transform of $Q_y(x)$ is given by  $(\SF_{\lambda}Q_y)(\xi)=-i({\rm
sgn}\,\xi)e^{-y|\xi|}$, for $\xi\neq0$, and
$
(\tau_xQ_y)(-t)=-ic_{\lambda}\int_{\RR}({\rm
sgn}\,\xi)e^{-y|\xi|}E_{\lambda}(ix\xi)E_{\lambda}(-it\xi)|\xi|^{2\lambda}d\xi.
$
\end{proposition}

Then we can define the associated maximal functions as
$$(P^*_{\nabla}f)(x)=\sup_{|s-x|<y}|(Pf)(s,y)|,\ \ \
(P^*f)(x)=\sup_{y>0}|(Pf)(x,y)|,$$
$$(Q^*_{\nabla}f)(x)=\sup_{|s-x|<y}|(Qf)(s,y)|,\ \ \ (Q^*f)(x)=\sup_{y>0}|(Qf)(x,y)|.$$

\begin{proposition}\cite{ZhongKai Li 3}\label{Poisson-conjugate-CR}
{\rm (i)}    For $f\in L^p_{\lambda}(\RR)$, $1\le p<\infty$,  $u(x,y)=(Pf)(x,y)$ and  $v(x,y)=(Qf)(x,y)$
on $\RR^2_+$ satisfy the $\lambda$-Cauchy-Riemann equations\,(\ref{a c r0}), and are both $\lambda$-harmonic on $\RR^2_+$.\\
{\rm (ii)}\rm(semi-group property)  If $f\in L^p_{\lambda}(\RR)$, $1\leq p\leq \infty$, and $y_0>0$, then
$(Pf)(x,y_0+y)=P[(Pf)(\cdot,y_0)](x,y), \ \hbox{for} \ y>0.$\\
{\rm (iii)}  If $1 < p < \infty$, then there exists some constant $A_p'$ for any $f \in
L^p_{\lambda}(\RR)$, $\|(Q^*_{\nabla})f\|_{L^p_{\lambda}} \le A_p'
\|f\|_{L^p_{\lambda}}$.\\
{\rm (iv)}   $P^*_{\nabla}$ and $P^*$ are both $(p,p)$ type for $1<p
\leq\infty$ and weak-$(1,1)$ type.\\
{\rm (v)}    If $d\mu\in {\frak B}_{\lambda}(\RR)$, then
$\|(P(d\mu))(\cdot,y)\|_{L_{\lambda}^1}\le\|d\mu\|_{{\frak
B}_{\lambda}}$ as $y\rightarrow0+$, $[P(d\mu)](\cdot,y)$
converges $*$-weakly to $d\mu$: If $f\in X=L_{\lambda}^p(\RR)$, $1\le
p<\infty$, or $C_0(\RR)$, then $\|(Pf)(\cdot,y)\|_X\le\|f\|_X$ and
$\lim_{y\rightarrow0+}\|(Pf)(\cdot,y)-f\|_{X}=0$.\\
{\rm (vi)}  If $f\in L^p_{\lambda}(\RR)$, $1\le p\le2$, and
$[\SF_{\lambda}(Pf(\cdot,y))](\xi)=e^{-y|\xi|}(\SF_{\lambda}f)(\xi)$,
and
\begin{eqnarray}\label{D-Poisson-2}
(Pf)(x,y)=c_{\lambda}\int_{\RR}e^{-y|\xi|}(\SF_{\lambda}f)(\xi)E_{\lambda}(ix\xi)|\xi|^{2\lambda}d\xi,\quad
(x,y)\in\RR^2_+,
\end{eqnarray}
furthermore, Formula\,(\ref{D-Poisson-2}) is true when we replace $f\in L^p_{\lambda}(\RR)$ with $d\mu\in {\frak
B}_{\lambda}(\RR)$.\\
{\rm (vii)} \    If $1\le p<\infty$ and  $F=u+iv\in H_{\lambda}^p(\RR^2_+)$, then $F$
is the $\lambda$-Poisson integral of its boundary values $F(x)$, and $F(x)\in
L^p_{\lambda}(\RR)$.\\
\end{proposition}

\begin{lemma}\cite{ZhongKai Li 3}\label{Hardy-thm-0} Let $F\in H^p_{\lambda}(\RR^2_+)$, $\frac{2\lambda}{2\lambda+1}\le p\le1$.
Then  there exists a function $\phi$ on $\RR$, satisfying

{\rm(i)} \  For all $(x,y)\in\RR^2_+$,
\begin{eqnarray}\label{Hp-represent-1}
F(x,y)=c_{\lambda}\int_0^{\infty}e^{-y|\xi|}\phi(\xi)E_{\lambda}(ix\xi)|\xi|^{2\lambda}d\xi;
\end{eqnarray}

{\rm(ii)} \ $\phi$ is continuous on $\RR$, and $\phi(\xi)=0$ for
$\xi\in(-\infty,0]$;

{\rm(iii)} \ For $y>0$, the function $\xi\mapsto
e^{-y|\xi|}\phi(\xi)$ is bounded on $\RR$ satisfying $\xi\mapsto
e^{-y|\xi|}\phi(\xi)\in L_{\lambda}^1(\RR)$;

{\rm(iv)} \ Obviously, the function $\phi$ satisfying the following inequality:
$$\sup_{y>0}c_{\lambda}\int_0^{\infty}\left|\int_0^{\infty}e^{-y|\xi|}\phi(\xi)E_{\lambda}(ix\xi)|\xi|^{2\lambda}d\xi\right|^p|x|^{2\lambda} dx<\infty.$$

Thus we could deduce that for $\frac{2\lambda}{2\lambda+1}\le p\le1$, $F\in H^p_{\lambda}(\RR^2_+)$ if and only if
there exits a $\phi$ satisfying {\rm(i)},\, {\rm(ii)},\, {\rm(iii)} ,\,{\rm(iv)}.

\end{lemma}
\begin{proof}
Notice that for the operator $T_{\bar{z}}=\frac{1}{2}(D_x+i\partial_y)$, with $\phi(\xi)$ satisfying {\rm(i)},\, {\rm(ii)},\, {\rm(iii)} ,\,{\rm(iv)},\,we have
\begin{eqnarray*}
T_{\bar{z}}c_{\lambda}\int_0^{\infty}e^{-y|\xi|}\phi(\xi)E_{\lambda}(ix\xi)|\xi|^{2\lambda}d\xi=c_{\lambda}\int_0^{\infty}\frac{1}{2}\big(-i\xi+i\xi\big)e^{-y|\xi|}\phi(\xi)E_{\lambda}(ix\xi)|\xi|^{2\lambda}d\xi=0.
\end{eqnarray*}
Thus  the function $c_{\lambda}\int_0^{\infty}e^{-y|\xi|}\phi(\xi)E_{\lambda}(ix\xi)|\xi|^{2\lambda}d\xi$ is $\lambda$-analytic on the upper half plane $\RR^2_+$.

\end{proof}

\begin{lemma}\label{estimation}

For $y>0$, $x,\,z\in\RR$, we have the following estimations for the $\lambda$-Poisson kernel $(\tau_xP_y)(-z)$ and the conjugate  $\lambda$-Poisson kernel $(\tau_xQ_y)(-z)$
\begin{eqnarray}\label{D-Poisson-ker-110}
\left|\partial_x(\tau_xP_y)(-z)\right|\lesssim\bigg(\frac{1}{y^2+(|x|-|z|)^2}\bigg)^{\lambda+1},\ \ \ \ \left|\partial_y(\tau_xP_y)(-z)\right|\lesssim\bigg(\frac{1}{y^2+(|x|-|z|)^2}\bigg)^{\lambda+1},
\end{eqnarray}
\begin{eqnarray}\label{D-Poisson-ker-112}
\left|\partial_y\partial_y(\tau_xP_y)(-z)\right|\lesssim\bigg(\frac{1}{y^2+(|x|-|z|)^2}\bigg)^{\lambda+\frac{3}{2}},\ \ \ \left|\partial_x\partial_x(\tau_xP_y)(-z)\right|\lesssim\bigg(\frac{1}{y^2+(|x|-|z|)^2}\bigg)^{\lambda+\frac{3}{2}},
\end{eqnarray}
\begin{eqnarray}\label{D-Poisson-ker-114}
\left|\partial_x(\tau_xQ_y)(-z)\right|\lesssim\bigg(\frac{1}{y^2+(|x|-|z|)^2}\bigg)^{\lambda+1},\ \ \ \left|\partial_y(\tau_xQ_y)(-z)\right|\lesssim\bigg(\frac{1}{y^2+(|x|-|z|)^2}\bigg)^{\lambda+1}
\end{eqnarray}
\begin{eqnarray}\label{D-Poisson-ker-116}
\left|\partial_y\partial_y(\tau_xQ_y)(-t)\right|\lesssim\bigg(\frac{1}{y^2+(|x|-|z|)^2}\bigg)^{\lambda+\frac{3}{2}},\ \ \ \left|\partial_x\partial_x(\tau_xQ_y)(-t)\right|\lesssim\bigg(\frac{1}{y^2+(|x|-|z|)^2}\bigg)^{\lambda+\frac{3}{2}}.
\end{eqnarray}

\end{lemma}

\begin{proof}
For the case when $z=0$ (or $x=0$),  we have $(\tau_xP_y)(0)=P_y(x)=m_{\lambda}y(y^2+x^2)^{-\lambda-1}$ and $(\tau_xQ_y)(0)=Q_y(x)=m_{\lambda}x(y^2+x^2)^{-\lambda-1}$, where
$m_\lambda=2^{\lambda+1/2}\Gamma(\lambda+1)/\sqrt{\pi}$. We could obtain the Formulas\,(\ref{D-Poisson-ker-110}, \ref{D-Poisson-ker-112}, \ref{D-Poisson-ker-114}, \ref{D-Poisson-ker-116}) directly.

Then, we will only consider the case when $xz>0$\,(The case $xz<0$ can be discussed in the same way ). By Proposition\,\ref{Poisson-a}, we could deduce that
\begin{eqnarray*}
\int_0^\pi\bigg|\frac{\partial}{\partial x}\frac{y(1+\cos\theta)\sin^{2\lambda-1}\theta
}{\big(y^2+x^2+z^2-2xz\cos\theta\big)^{\lambda+1}}\bigg|
d\theta&=&\int_0^\pi\bigg|\frac{y(\lambda+1)(2x-2z\cos\theta)(1+\cos\theta)
}{\big(y^2+x^2+z^2-2xz\cos\theta\big)^{\lambda+2}}\sin^{2\lambda-1}\theta\bigg|
d\theta
\\ \nonumber&\leq& C \bigg(\frac{1}{y^2+(|x|-|z|)^2}\bigg)^{\lambda+1}.
\end{eqnarray*}
Thus
\begin{eqnarray}\label{a1}
\left|\partial_x(\tau_xP_y)(-z)\right|\lesssim\bigg(\frac{1}{y^2+(|x|-|z|)^2}\bigg)^{\lambda+1}.
\end{eqnarray}

By Proposition\,\ref{Poisson-a}, we could also deduce that
\begin{eqnarray*}
& &\int_0^\pi\bigg|\frac{\partial}{\partial y}\frac{y(1+\cos\theta)\sin^{2\lambda-1}\theta
}{\big(y^2+x^2+z^2-2xz\cos\theta\big)^{\lambda+1}}\bigg|
d\theta
\\ \nonumber &\leq&\int_0^\pi\bigg|(\lambda+1)\frac{y^2(1+\cos\theta)\sin^{2\lambda-1}\theta
}{\big(y^2+x^2+z^2-2xz\cos\theta\big)^{\lambda+2}}\bigg|
d\theta+\int_0^\pi\bigg|\frac{(1+\cos\theta)\sin^{2\lambda-1}\theta
}{\big(y^2+x^2+z^2-2xz\cos\theta\big)^{\lambda+1}}\bigg|
d\theta
\\ \nonumber&\leq& C \bigg(\frac{1}{y^2+(|x|-|z|)^2}\bigg)^{\lambda+1}.
\end{eqnarray*}
Thus
\begin{eqnarray}\label{a2}
\left|\partial_y(\tau_xP_y)(-z)\right|\lesssim\bigg(\frac{1}{y^2+(|x|-|z|)^2}\bigg)^{\lambda+1}.
\end{eqnarray}

In the same way as Formulas\,(\ref{a1},\,\ref{a2}), by Proposition\,\ref{Poisson-a} and Proposition\,\ref{conjugate-Poisson-a}, we could obtain
Formulas\,(\ref{D-Poisson-ker-110}, \ref{D-Poisson-ker-112}, \ref{D-Poisson-ker-114}, \ref{D-Poisson-ker-116}) for the case $xz>0$. This proves the Lemma.

\end{proof}

Thus by the Lemma\,\ref{estimation}, we could obtain the following Lemma:
\begin{lemma}\label{estimation2}

For $y>0$, $t>0$, $x,\,u\in\RR$, we have the following estimations for the $\lambda$-Poisson kernel $(\tau_{t^{-1}x}P_{t^{-1}y})(-u)$ and the conjugate  $\lambda$-Poisson kernel $(\tau_{t^{-1}x}Q_{t^{-1}y})(-u)$
\begin{eqnarray*}
& &\left|\partial_{x}(\tau_{t^{-1}x}P_{t^{-1}y})(-u)\right|+\left|\partial_y(\tau_{t^{-1}x}P_{t^{-1}y})(-u)\right|\lesssim t^{-1}\bigg(\frac{1}{(t^{-1}y)^2+(|t^{-1}x|-|u|)^2}\bigg)^{\lambda+1},
\end{eqnarray*}
\begin{eqnarray*}
& &\left|\partial_y\partial_y(\tau_{t^{-1}x}P_{t^{-1}y})(-u)\right|+\left|\partial_x\partial_x(\tau_{t^{-1}x}P_{t^{-1}y})(-u)\right|\lesssim t^{-2}\bigg(\frac{1}{(t^{-1}y)^2+(|t^{-1}x|-|u|)^2}\bigg)^{\lambda+\frac{3}{2}},
\end{eqnarray*}
\begin{eqnarray*}
& &\left|\partial_x(\tau_{t^{-1}x}Q_{t^{-1}y})(-u)\right|+\left|\partial_y(\tau_{t^{-1}x}Q_{t^{-1}y})(-u)\right|\lesssim t^{-1}\bigg(\frac{1}{(t^{-1}y)^2+(|t^{-1}x|-|u|)^2}\bigg)^{\lambda+1},
\end{eqnarray*}
\begin{eqnarray*}
& &\left|\partial_y\partial_y(\tau_{t^{-1}x}Q_{t^{-1}y})(-u)\right|+\left|\partial_x\partial_x(\tau_{t^{-1}x}Q_{t^{-1}y})(-u)\right|\lesssim t^{-2}\bigg(\frac{1}{(t^{-1}y)^2+(|t^{-1}x|-|u|^2}\bigg)^{\lambda+\frac{3}{2}}.
\end{eqnarray*}

\end{lemma}

\subsection{The boundedness of Ces\`{a}ro operator on $H_{\lambda}^p(\RR_+^2)$ }\label{ce Hp4}

\begin{proposition}\label{cesa6}Ces\`{a}ro operator $C_{\alpha}$ defined on $H_{\lambda}^p(\RR_+^2)$ is a bounded linear operator on $H_{\lambda}^p(\RR_+^2)$ for $1\leq p<\infty$:
$$ \|C_{\alpha}F\|_{H_{\lambda}^p(\RR_+^2)}\leq C\|F\|_{H_{\lambda}^p(\RR_+^2)},$$
where C is dependent on $\alpha$, $p$ and $\lambda$.
\end{proposition}
\begin{proof}
Let $F(x, y)\in H_{\lambda}^p(\RR_+^2)$\,($1\le p<\infty$).  By Proposition\,\ref{Poisson-conjugate-CR}\,{\rm (vii)}, $F(x, y)$
is the $\lambda$-Poisson integral of its boundary values. Let $f(x)$ to be the real parts of the boundary values of  $F(x, y)$ and we could see that $f(x)\in
L^p_{\lambda}(\RR)$. By Proposition\,\ref{Poisson-conjugate-CR} {\rm (i)}  and {\rm (vii)}, we could deduce that
$$F(x, y)=f\ast_\lambda P_y(x)+if\ast_\lambda Q_y(x).$$
Thus we could have
\begin{eqnarray}\label{se}
C_{\alpha}F(x, y)&=&\int_0^1 f\ast_\lambda P_{y/t}(x/t)t^{-1}\phi_\alpha(t)dt+i\int_0^1f\ast_\lambda Q_{y/t}(x/t)t^{-1}\phi_\alpha(t)dt
\\ \nonumber&=&\int_0^1 \bigg(\int_\RR f(u)\tau_{-u}P_{t^{-1}y}(t^{-1}x)|u|^{2\lambda}du\bigg)t^{-1}\phi_\alpha(t)dt
\\ \nonumber&+&i\int_0^1 \bigg(\int_\RR f(u)\tau_{-u}Q_{t^{-1}y}(t^{-1}x)|u|^{2\lambda}du\bigg)t^{-1}\phi_\alpha(t)dt.
\end{eqnarray}
Notice that we could deduce the following Formula\,(\ref{2}) by  Proposition\,\ref{tau-convolution-a}\,{\rm(i)}, {\rm(ii)} and Holder inequality:
\begin{eqnarray}\label{2}
& &\int_0^1 \bigg(\int_\RR
\bigg|f(u)\partial_y\tau_{-u}P_{t^{-1}y}(t^{-1}x)\bigg||u|^{2\lambda}du\bigg)t^{-1}\phi_\alpha(t)dt
\\ \nonumber&\leq&\int_0^1 \bigg(\int_\RR \left|f(u)\right|^p|u|^{2\lambda}du\bigg)^{1/p}\bigg(\int_\RR \bigg|\partial_y\tau_{t^{-1}x}P_{t^{-1}y}(-u)\bigg|^q|u|^{2\lambda}du\bigg)^{1/q}t^{-1}\phi_\alpha(t)dt,
\end{eqnarray}
where $p,\ q$ satisfy $$p\geq1,\  q\geq1,\  \frac{1}{p}+\frac{1}{q}=1.$$
By Lemma\,\ref{estimation2} and Formula\,(\ref{2}), we could deduce the following inequality:
\begin{eqnarray}\label{2*}
& &\int_0^1 \bigg(\int_\RR
\bigg|f(u)\partial_y\tau_{-u}P_{t^{-1}y}(t^{-1}x)\bigg||u|^{2\lambda}du\bigg)t^{-1}\phi_\alpha(t)dt
\\ \nonumber&\leq&\int_0^1 \bigg(\int_\RR \left|f(u)\right|^p|u|^{2\lambda}du\bigg)^{1/p}\bigg(\int_\RR \bigg|t^{-1}\bigg(\frac{1}{(t^{-1}y)^2+(|t^{-1}x|-|u|)^2}\bigg)^{\lambda+1}\bigg|^q|u|^{2\lambda}du\bigg)^{1/q}t^{-1}\phi_\alpha(t)dt.
\end{eqnarray}
Let $u=t^{-1}z$, then we could deduce the following inequality by Formula\,(\ref{2*})
\begin{eqnarray}\label{2**}
& &\int_0^1 \bigg(\int_\RR
\bigg|f(u)\partial_y\tau_{-u}P_{t^{-1}y}(t^{-1}x)\bigg||u|^{2\lambda}du\bigg)t^{-1}\phi_\alpha(t)dt
\\ \nonumber&\lesssim&\int_0^1 \bigg(\int_\RR \left|f(u)\right|^p|u|^{2\lambda}du\bigg)^{1/p}\bigg(\int_\RR \bigg|t^{-1}\bigg(\frac{1}{(t^{-1}y)^2+(|t^{-1}x|-|t^{-1}z|)^2}\bigg)^{\lambda+1}\bigg|^qt^{-(2\lambda+1)}|z|^{2\lambda}dz\bigg)^{1/q}t^{-1}\phi_\alpha(t)dt
\\ \nonumber&\lesssim&\bigg(\int_\RR \left|f(u)\right|^p|u|^{2\lambda}du\bigg)^{1/p}\bigg(\int_\RR \bigg|\bigg(\frac{1}{(y)^2+(|x|-|z|)^2}\bigg)^{\lambda+1}\bigg|^q|z|^{2\lambda}dz \bigg)^{1/q}\int_0^1t^{2\lambda-\frac{2\lambda+1}{q}}\phi_\alpha(t)dt
\\ \nonumber&<&\infty.
\end{eqnarray}
Thus similar to Formula\,(\ref{2**}), by Lemma\,\ref{estimation2}, we could obtain
\begin{eqnarray}\label{3**}
& &\int_0^1 \bigg(\int_\RR
\bigg|f(u)\partial_x\tau_{-u}P_{t^{-1}y}(t^{-1}x)\bigg||u|^{2\lambda}du\bigg)t^{-1}\phi_\alpha(t)dt
\\ \nonumber&\lesssim&\bigg(\int_\RR \left|f(u)\right|^p|u|^{2\lambda}du\bigg)^{1/p}\bigg(\int_\RR \bigg|\bigg(\frac{1}{(y)^2+(|x|-|z|)^2}\bigg)^{\lambda+1}\bigg|^q|z|^{2\lambda}dz \bigg)^{1/q}\int_0^1t^{2\lambda-\frac{2\lambda+1}{q}}\phi_\alpha(t)dt
\\ \nonumber&<&\infty.
\end{eqnarray}

\begin{eqnarray}\label{4**}
& &\int_0^1 \bigg(\int_\RR
\bigg|f(u)(\partial_y\partial_y+\partial_x\partial_x)\tau_{-u}P_{t^{-1}y}(t^{-1}x)\bigg||u|^{2\lambda}du\bigg)t^{-1}\phi_\alpha(t)dt
\\ \nonumber&\lesssim&\bigg(\int_\RR \left|f(u)\right|^p|u|^{2\lambda}du\bigg)^{1/p}\bigg(\int_\RR \bigg|\bigg(\frac{1}{(y)^2+(|x|-|z|)^2}\bigg)^{\lambda+\frac{3}{2}}\bigg|^q|z|^{2\lambda}dz \bigg)^{1/q}\int_0^1t^{2\lambda-\frac{2\lambda+1}{q}}\phi_\alpha(t)dt
\\ \nonumber&<&\infty.
\end{eqnarray}
In the same way, we could also deduce the following inequalities\,(\ref{5**},\,\ref{6**},\,\ref{7**}):
\begin{eqnarray}\label{5**}
\int_0^1 \bigg(\int_\RR
\bigg|f(u)(\partial_x)\tau_{-u}Q_{t^{-1}y}(t^{-1}x)\bigg||u|^{2\lambda}du\bigg)t^{-1}\phi_\alpha(t)dt&<&\infty,
\end{eqnarray}

\begin{eqnarray}\label{6**}
\int_0^1 \bigg(\int_\RR
\bigg|f(u)(\partial_y)\tau_{-u}Q_{t^{-1}y}(t^{-1}x)\bigg||u|^{2\lambda}du\bigg)t^{-1}\phi_\alpha(t)dt&<&\infty,
\end{eqnarray}

and
\begin{eqnarray}\label{7**}
\int_0^1 \bigg(\int_\RR
\bigg|f(u)(\partial_y\partial_y+\partial_x\partial_x)\tau_{-u}Q_{t^{-1}y}(t^{-1}x)\bigg||u|^{2\lambda}du\bigg)t^{-1}\phi_\alpha(t)dt&<&\infty.
\end{eqnarray}

From inequalities\,(\ref{2**},\,\ref{3**},\,\ref{4**},\,\ref{5**},\,\ref{6**},\,\ref{7**}), we could take $\triangle_{\lambda}$, $\partial_y$ and $D_x$ under integration signs of
Formula\,(\ref{se}).
Notice that we have the following three Formulas holds
$$
(\tau_xP_y)(-u)=c_{\lambda}\int_{\RR}e^{-y|\xi|}E_{\lambda}(ix\xi)E_{\lambda}(-iu\xi)|\xi|^{2\lambda}d\xi,
$$

$$
(\tau_xQ_y)(-u)=-ic_{\lambda}\int_{\RR}({\rm
sgn}\,\xi)e^{-y|\xi|}E_{\lambda}(ix\xi)E_{\lambda}(-iu\xi)|\xi|^{2\lambda}d\xi,
$$
\begin{eqnarray*}\label{D-kernel-2*}
D_x[E_{\lambda}(ix\xi)]=i\xi E_{\lambda}(ix\xi), \ \ \ \ \ \hbox{and} \
\ \ \ E_{\lambda}(ix\xi)|_{x=0}=1.
\end{eqnarray*}
Then we take $\triangle_{\lambda}$, $\partial_y$ and $D_x$ under integration signs of
Formula\,(\ref{se}), and we could deduce that the following holds:
\begin{eqnarray*}
T_{\bar{z}}C_{\alpha}F(x, y)&=&T_{\bar{z}}\int_0^1 f\ast_\lambda P_{y/t}(x/t)t^{-1}\phi_\alpha(t)dt+iT_{\bar{z}}\int_0^1f\ast_\lambda Q_{y/t}(x/t)t^{-1}\phi_\alpha(t)dt
\\ \nonumber &=&T_{\bar{z}}\int_0^1 \bigg(\int_\RR f(u)\tau_{-u}P_{t^{-1}y}(t^{-1}x)|u|^{2\lambda}du\bigg)t^{-1}\phi_\alpha(t)dt
\\ \nonumber&+&iT_{\bar{z}}\int_0^1 \bigg(\int_\RR f(u)\tau_{-u}Q_{t^{-1}y}(t^{-1}x)|u|^{2\lambda}du\bigg)t^{-1}\phi_\alpha(t)dt
\\ \nonumber &=&c_{\lambda}\int_0^1 \int_\RR f(u)\bigg(T_{\bar{z}}\int_{\RR}e^{-t^{-1}y|\xi|}E_{\lambda}(it^{-1}x\xi)E_{\lambda}(-iu\xi)|\xi|^{2\lambda}d\xi\bigg)|u|^{2\lambda}dut^{-1}\phi_\alpha(t)dt
\\ \nonumber&-&c_{\lambda}\int_0^1 \int_\RR f(u)\bigg(T_{\bar{z}}\int_{\RR}({\rm
sgn}\,\xi)e^{-t^{-1}y|\xi|}E_{\lambda}(it^{-1}x\xi)E_{\lambda}(-iu\xi)|\xi|^{2\lambda}d\xi\bigg)|u|^{2\lambda}dut^{-1}\phi_\alpha(t)dt
\\ \nonumber &=&2c_{\lambda}\int_0^1 \int_\RR f(u)\bigg(T_{\bar{z}}\int_0^{+\infty}e^{-t^{-1}y|\xi|}E_{\lambda}(it^{-1}x\xi)E_{\lambda}(-iu\xi)|\xi|^{2\lambda}d\xi\bigg)|u|^{2\lambda}dut^{-1}\phi_\alpha(t)dt
\\ \nonumber&=&0.
\end{eqnarray*}

Thus $(C_{\alpha}F)(x, y)$ is a $\lambda$-analytic function, together with Theorem\,\ref{cesar2}, we could deduce that Ces\`{a}ro operator $C_{\alpha}$ defined on $H_{\lambda}^p(\RR_+^2)$ is a bounded linear operator on $H_{\lambda}^p(\RR_+^2)$ for $1\leq p<\infty$:
$$ \|C_{\alpha}F\|_{H_{\lambda}^p(\RR_+^2)}\leq C\|F\|_{H_{\lambda}^p(\RR_+^2)},$$
where C is a constant dependent on $\alpha$, $p$ and $\lambda$.
\end{proof}

\begin{proposition}\label{cesa4}Ces\`{a}ro operator $C_{\alpha}$ defined on $H_{\lambda}^p(\RR_+^2)$ is a bounded linear operator on $H_{\lambda}^p(\RR_+^2)$ for $\frac{2\lambda}{2\lambda+1}< p\leq1$:
$$ \|C_{\alpha}f\|_{H_{\lambda}^p(\RR_+^2)}\leq C\|f\|_{H_{\lambda}^p(\RR_+^2)}.$$
C is dependent on $\alpha$, $p$ and $\lambda$.
\end{proposition}
\begin{proof}
From Lemma\,\ref{Hardy-thm-0}, we could deduce that for $F(x, y)\in H_{\lambda}^p(\RR_+^2)$\,$(\frac{2\lambda}{2\lambda+1}< p\leq1)$, if and only if there exists a function $\phi(\xi)$ satisfying {\rm(i)},\, {\rm(ii)},\, {\rm(iii)} ,\,{\rm(iv)} in Lemma\,\ref{Hardy-thm-0}. We define the operator $B_\alpha$ as
following:
$$B_\alpha\phi(\xi)=\int_0^1\phi(t\xi)\phi_\alpha(t)|t|^{2\lambda}dt.$$
Let $B(x, y)$ to denote as the function
$$B(x, y)=c_{\lambda}\int_0^{\infty}e^{-y|\xi|}B_\alpha\phi(\xi)E_{\lambda}(ix\xi)|\xi|^{2\lambda}d\xi.$$
Notice that for $0<t\leq1$, $\big|e^{-y|\xi|}\phi(t\xi)\big|=\big|e^{-y|\xi|/2}\big(e^{-y|\xi|/2}\phi(t\xi)\big)\big|<\big|e^{-y|\xi|/2}\big(e^{-ty|\xi|/2}\phi(t\xi)\big)\big|.$ By Lemma\,\ref{Hardy-thm-0},
the function $e^{-ty|\xi|/2}\phi(t\xi)$ is bounded on $\RR$, thus $e^{-y|\xi|}\phi(t\xi)\in L_{\lambda}^1(\RR)$. Then by Fubini theorem, we could write $B(x, y)$ as:
\begin{eqnarray}\label{ast1}
B(x, y)&=&c_{\lambda}\int_0^{\infty}e^{-y|\xi|}B_\alpha\phi(\xi)E_{\lambda}(ix\xi)|\xi|^{2\lambda}d\xi
\nonumber\\&=&c_{\lambda}\int_0^{\infty}e^{-y|\xi|}\int_0^1\phi(t\xi)\phi_\alpha(t)|t|^{2\lambda}E_{\lambda}(ix\xi)|\xi|^{2\lambda}dtd\xi
\nonumber\\&=&c_{\lambda}\int_0^1\int_0^{\infty}e^{-y|\xi|}\phi(t\xi)\phi_\alpha(t)|t|^{2\lambda}E_{\lambda}(ix\xi)|\xi|^{2\lambda}d\xi dt.
 \end{eqnarray}

Notice that
\begin{eqnarray}\label{ast2}
\frac{1}{t^{2\lambda+1}}F(t^{-1}x, t^{-1}y)=c_{\lambda}\int_0^{\infty}e^{-y|\xi|}\phi(t\xi)E_{\lambda}(ix\xi)|\xi|^{2\lambda}d\xi,
 \end{eqnarray}
where $F(x, y)\in  H_{\lambda}^p(\RR_+^2)$ for $\frac{2\lambda}{2\lambda+1}< p\leq1$,  and that $F(x, y)$ and $\phi(\xi)$ satisfy {\rm(i)},\, {\rm(ii)},\, {\rm(iii)} ,\,{\rm(iv)} in Lemma\,\ref{Hardy-thm-0}:\  $F(x, y)=c_{\lambda}\int_0^{\infty}e^{-y|\xi|}\phi(\xi)E_{\lambda}(ix\xi)|\xi|^{2\lambda}d\xi.$
Thus by Formulas\,(\ref{ast1},\,\ref{ast2}), we could have
$$B(x, y)=\int_0^1 F(t^{-1}x, t^{-1}y)t^{-1}\phi_\alpha(t)dt =(C_\alpha F)(x, y).$$
By Theorem\,\ref{cesar2}, we could deduce that
$$ \|C_{\alpha}F\|_{L^p(\RR^{2}_+, |x|^{2\lambda }dx))}\leq C\|F\|_{L^p(\RR^{2}_+, |x|^{2\lambda }dx)}.$$
It is easy to check that the function $B_\alpha\phi(\xi)$ satisfies {\rm(i)},\, {\rm(ii)},\, {\rm(iii)} ,\,{\rm(iv)} in Lemma\,\ref{Hardy-thm-0}, thus we could deduce that
$B(x, y)=(C_\alpha F)(x, y)\in H_{\lambda}^p(\RR_+^2)$. Together with Theorem\,\ref{cesar2}, we could deduce the Propositon\,\ref{cesa4} that Ces\`{a}ro operator $C_{\alpha}$ defined on $H_{\lambda}^p(\RR_+^2)$ is a bounded linear operator on $H_{\lambda}^p(\RR_+^2)$ for $\frac{2\lambda}{2\lambda+1}< p\leq1$.
\end{proof}
By Proposition\,\ref{cesa6} and Proposition\,\ref{cesa4}, we could obtain the following theorem:
\begin{theorem}\label{cesa0}Ces\`{a}ro operator $C_{\alpha}$ defined on $H_{\lambda}^p(\RR_+^2)$ is a bounded linear operator on $H_{\lambda}^p(\RR_+^2)$ for $\frac{2\lambda}{2\lambda+1}< p<\infty$:
$$ \|C_{\alpha}f\|_{H_{\lambda}^p(\RR_+^2)}\leq C\|f\|_{H_{\lambda}^p(\RR_+^2)},$$
where C is a constant dependent on $\alpha$, $p$ and $\lambda$.
\end{theorem}

\textbf{Declarations}:

\textbf{Availability of data and materials}:Not applicable

\textbf{Competing interests}: The authors declare that they have no competing interests

\textbf{Funding}:Not applicable

\textbf{Authors’ contributions}:The author ZhuoRan Hu finished this paper alone.

\textbf{Acknowledgements}:Not applicable

\textbf{Author}: ZhuoRan Hu

\end{document}